\documentclass[a4paper,draft,reqno,12pt]{amsart}
\usepackage[english]{babel}
\usepackage{amsmath}
\usepackage{amssymb}
\usepackage{amscd}
\usepackage{amsthm}
\usepackage{euscript}
\usepackage[all]{xy}

\newtheorem{problem}{Problem}
\newtheorem{lemma}{Lemma}
\newtheorem{theorem}{Theorem}

\theoremstyle{definition}
\newtheorem{definition}{Definition}
\newtheorem{example}{Example}
\theoremstyle{remark}
\newtheorem {remark}{Remark}

\DeclareMathOperator{\Spec}{Spec}

\DeclareMathOperator{\reg}{reg}

\DeclareMathOperator{\Cl}{Cl}

\def\GG{{\mathbb G}}

\def\CC{{\mathbb C}}

\def\ZZ{{\mathbb Z}}

\def\PP{{\mathbb P}}
\def\AA{{\mathbb A}}

\begin{document}
\date{}
\title[Polynomial curves on trinomial hypersurfaces]{Polynomial curves on trinomial hypersurfaces}
\author{Ivan Arzhantsev}
\thanks{The research was supported by the grant RSF-DFG 16-41-01013}
\address{National Research University Higher School of Economics, Faculty of Computer Science, Kochnovskiy Proezd 3, Moscow, 125319 Russia}
\email{arjantsev@hse.ru}



\maketitle

\begin{abstract}
We prove that every rational trinomial affine hypersurface admits a horizontal polynomial curve. This result provides an explicit non-trivial polynomial solution to a trinomial equation. Also we show that a trinomial affine hypersurface admits a Schwartz-Halphen curve if and only if the trinomial comes from a platonic triple. It is a generalization of Schwartz-Halphen's Theorem for Pham-Brieskorn surfaces.
\end{abstract}

\section{Introduction}
It is well known that the Fermat equation $z_0^p+z_1^p+z_2^p=0$, $p\ge 3$, has no non-trivial solution over the polynomial ring $\CC[x]$. The reason for this is that the projective curve defined by the Fermat equation in $\PP^2$ is not rational.

It is natural to consider more general equations
\begin{equation}
z_0^p+z_1^q+z_2^r=0, \quad p,q,r \in\ZZ_{\ge 2},
\label{PB}
\end{equation}
and to ask for polynomial solutions. Geometrically such a solution corresponds to a polynomial curve $\tau\colon\CC\to V_{p,q,r}$, where $V_{p,q,r}:=V(z_0^p+z_1^q+z_2^r)$ is called the Pham-Brieskorn surface in~$\CC^3$. Here we have trivial solutions, namely,
$$
z_0(x)=\alpha\phi(x)^{m/p}, \quad
z_1(x)=\beta\phi(x)^{m/q}, \quad
z_2(x)=\gamma\phi(x)^{m/r},
$$
where $m=\text{lcm}(p,q,r)$, $\phi(x)\in\CC[x]$, and $\alpha,\beta,\gamma\in\CC$ with
$\alpha^p+\beta^q+\gamma^r=0$.

\smallskip

The following result is stated in \cite[Theorem~0.1~(a)]{FZ} with references to \cite{BK}, \cite{Ev} and \cite[Corollary of Lemma~8]{KZ}. 

\begin{theorem} \label{T1}
The Pham-Brieskorn surface $V_{p,q,r}$ admits a non-trivial polynomial curve if and only if one of the following conditions hold.
\begin{enumerate}
\item[(i)]
At least one of the numbers $p,q,r$ is coprime with the others.
\item[(ii)]
We have $\text{gcd}(p,q)=\text{gcd}(p,r)=\text{gcd}(q,r)=2$.
\end{enumerate}
\end{theorem}

Moreover, the conditions of Theorem~\ref{T1} characterize rational  Pham-Brieskorn surfaces, see~\cite[p.~117]{BK}.

\smallskip

Now we come to a special class of non-trivial polynomial curves on $V_{p,q,r}$. Let us recall that a triple of positive integers $(p,q,r)$ is called \emph{platonic}
if we have $1/p+1/q+1/r>1$. It is well known that the platonic triples up to renumbering are the following ones
$$
(5,3,2), \quad (4,3,2), \quad (3,3,2), \quad (p,2,2), \quad (p,q,1), \quad p,q\in\ZZ_{>0}.
$$

In 1873, Schwartz~\cite{Sch} found polynomial solutions of equation~(\ref{PB}) in coprime polynomials $z_0(x), z_1(x), z_2(x)$ for every platonic triple $(p,q,r)$ with
$p,q,r\ge 2$; see also \cite{We} and \cite{FZ} for explicit formulas. 
 
In 1883, Halphen~\cite{Hal} proved that equaition~(\ref{PB}) has no solution in non-constant coprime polynomials when $1/p+1/q+1/r\le 1$. We refer to~\cite{Kl} 
for a historical account on the subject. 

\smallskip

Following \cite[Theorem 0.1 (b)]{FZ}, we reformulate these results in terms of polynomial curves. 

\begin{theorem} \label{T2}
The Pham-Brieskorn surface $V_{p,q,r}$ admits a polynomial curve not passing through the origin if and only if $(p,q,r)$ is a platonic triple.
\end{theorem}

There are several ways to generalize the theory of  Pham-Brieskorn surfaces to higher dimension. One way is to consider Pham-Brieskorn hypersurfaces 
$$
V(z_0^{p_0}+z_1^{p_1}+\ldots+z_m^{p_m})\subseteq\CC^{m+1},
$$
see \cite[Example~2.21]{FZ} and references therein for related results. 

\smallskip

In this paper we consider trinomial hypersurfaces of arbitrary dimension. Trinomial relations in many variables arise naturally in connection with torus actions of complexity one, multigraded algebras and Cox rings of algebraic varieties, see~\cite{ADHL,HH,HHS,HS,HW,HW2}.

 In Section~\ref{sec2} we introduce two types of trinomial affine hypersurfaces, discuss their geometric properties and define a torus action of complexity one for hypersurfaces of each type.  
Theorems~\ref{Hor} and~\ref{Hor1} are generalizations of Theorem~\ref{T1} to the case of trinomial hypersurfaces. It turns out that for hypersurfaces of Type~2 rationality is equivalent to existence of a non-trivial polynomial curve, while for Type~1 this is not the case. 

In Section~\ref{sec5} we define Schwartz-Halphen curves on trinomial hypersurfaces and study basic properties of such curves. An extension of Theorem~\ref{T2} to the hypersurface case is given in Theorem~\ref{SH}.
As one may expect, a significant role in our argumets plays the Mason-Stothers abc-Theorem.

\smallskip

The author is grateful to J\"urgen Hausen and Milena Wrobel for useful discussions.

\section{Preliminaries}
\label{sec2} 

In this section we introduce two types of trinomials over the field $\CC$ of complex numbers, cf.~\cite{HW},~\cite{HW2}.  

\smallskip

{\it Type~1.}\ We fix positive integers $n_1,n_2$ and let $n=n_1+n_2$. For each $i=1,2$, we take a tuple $l_i\in\ZZ_{>0}^{n_i}$ and define a monomial
$$
T_i^{l_i}:=T_{i1}^{l_{i1}}\ldots T_{in_i}^{l_{in_i}}\in\CC[T_{ij}; \ i=1,2, \, j=1,\ldots,n_i].
$$
By a \emph{trinomial of Type~1} we mean a polynomial of the form $T_1^{l_1}+T_2^{l_2}+1$. A
\emph{trinomial hypersurface of Type~1} is the zero set

$$
X=V(T_1^{l_1}+T_2^{l_2}+1)\subseteq\CC^n.
$$

It is easy to check that $X$ is an irreducible smooth affine variety of dimension $n-1$.

\smallskip

{\it Type~2.}\ Fix positive integers $n_0,n_1,n_2$ and let $n=n_0+n_1+n_2$. For each $i=0,1,2$, fix a tuple $l_i\in\ZZ_{>0}^{n_i}$ and define a monomial
$$
T_i^{l_i}:=T_{i1}^{l_{i1}}\ldots T_{in_i}^{l_{in_i}}\in\CC[T_{ij}; \ i=0,1,2, \, j=1,\ldots,n_i].
$$
By \emph{a trinomial of Type~2} we mean a polynomial of the form $T_0^{l_0}+T_1^{l_1}+T_2^{l_2}$. A \emph{trinomial hypersurface of Type~2} is

$$
X=V(T_0^{l_0}+T_1^{l_1}+T_2^{l_2})\subseteq\CC^n.
$$

One can check that $X$ is an irreducible normal affine variety of dimension $n-1$. Clearly, every trinomial surface of Type~2 is either the Pham-Brieskorn surface $V_{p,q,r}$ or is isomorphic to the affine plane $\CC^2$.

\smallskip

The following simple lemma describes the singular locus of $X$.

\begin{lemma} \label{aa}
A point  $(t_{01},\ldots,t_{2n_2})$ on a trinomial hypersurface $X$ of Type~2 is singular if and only if for every $i=0,1,2$  either there exist  $1\le j<k\le n_i$ with $t_{ij}=t_{ik}=0$, or we have $t_{ij}=0$ for some $1\le j\le n_i$ with $l_{ij}\ge 2$.
\end{lemma}

\begin{proof}
A point $x\in X$ is singular if and only if
$\frac{\partial(T_0^{l_0}+T_1^{l_1}+T_2^{l_2})}{\partial T_{ij}}(x)=0$ for all $i=0,1,2$ and
all $1\le j\le n_i$. This implies the assertion.
\end{proof}

Recall that the \emph{complexity} of an effective action $T\times X \to X$ of an algebraic torus $T$ on an irreducible algebraic variety $X$ is defined as $\dim X - \dim T$. Trinomial hypersurfaces of both types are equipped with a torus action of complexity one. Namely,
assume that every variable $T_{ij}$ is an eigenvector of a weight $w_{ij}$ with respect to a $T$-action. Then we have relations
$$
\sum_{j=1}^{n_1} l_{1j}w_{1j}=\sum_{j=1}^{n_2} l_{2j}w_{2j}=0
$$
for Type~1 and relations
$$
\sum_{j=1}^{n_0} l_{0j}w_{0j}=\sum_{j=1}^{n_1} l_{1j}w_{1j}=\sum_{j=1}^{n_2} l_{2j}w_{2j}
$$
for Type~2. There relations define a subgroup in the torus of all invertible diagonal matrices on $\CC^n$ whose connected component is a subtorus $T$ of codimension $2$, and the restricted action $T\times X\to X$ is effective.

For Type~1, the monomials $T_1^{l_1}$ and $T_2^{l_2}$ are non-constant regular invariants of the $T$-action. On the contrary, for a trinomial hypersurface $X$ of Type~2, every $T$-orbit on $X$ contains the origin in its closure, and thus every regular $T$-invariant is a constant.

\begin{example}
On the hypersurface
$$
X=V(T_{01}^2T_{02}^4+T_{11}^6+T_{21}^8)\subseteq\CC^4
$$
we have a $(\CC^{\times})^2$-action given by
$$
(t_1,t_2)\cdot(T_{01},T_{02},T_{11},T_{21})=(t_1^{12}t_2^{-2}T_{01},t_2T_{02},t_1^4T_{11},t_1^3T_{21}). 
$$
\end{example}

\section{Horizontal curves on trinomial hypersurfaces of Type~2}
\label{sec3} 

\begin{definition}
A \emph{polynomial curve} on an algebraic variety $X$ is a regular non-constant morphism
$\tau\colon\CC\to X$.
\end{definition}

Assume that a variety $X$ is affine and carries an action $T\times X\to X$ of an algebraic torus~$T$. Every
one-parameter subgroup $\gamma\colon\CC^{\times}\to T$ and every point $x_1\in X$ with a non-closed orbit $\gamma(\CC^{\times})\cdot x_1$ define a polynomial curve
$$
\tau\colon\CC\to X, \quad \tau(t)=\gamma(t)\cdot x_1 \quad \text{for all} \quad t\ne 0 \quad
\text{and} \quad \tau(0)=x_0,
$$
where $x_0$ is the limit point of the non-closed orbit $\gamma(\CC^{\times})\cdot x_1$. Our aim now is to define and to study a class of polynomial curves which in a sense is complementary to this class of curves. The following definition is a special case of the standard notion of a quasisection; see~\cite[Section~2.5]{PV}.

\begin{definition}
A polynomial curve $\tau\colon\CC\to X$ on an irreducible $T$-variety $X$ is called \emph{horizontal} if there exists a $T$-invariant open subset $W$ in $X$ such that $\tau(\CC)$ intersects all $T$-orbits on $W$.
\end{definition}

In the case of the Pham-Brieskorn surface $X=V_{p,q,r}$, every polynomial curve on $X$ is either horizontal or a closure of a $T$-orbit on $X$. Curves of the latter type correspond to trivial polynomial solutions mentioned in the Introduction.

\begin{lemma} \label{l01} 
A polynomial curve $\tau\colon\CC\to X$ on a trinomial hypersurface of Type~2 is horizontal if and only if the rational function $T_0^{l_0}/T_1^{l_1}$ is non-constant along
the image $\tau(\CC)$.
\end{lemma}

\begin{proof}
Assume that some coordinate function $T_{ij}$ vanishes on $\tau(\CC)$. Then the monomial $T_i^{l_i}$ is zero and the two remaining monomials in the trinomial relation are proportional
along~$\tau(\CC)$. At the same time, the image $\tau(\CC)$ is contained in a proper closed $T$-invariant subset $V(T_{ij})$ and thus the curve can not intersect generic $T$-orbits on $X$.

Hence we may assume that every coordinate $T_{ij}$ has finitely many zeroes on $\tau(\CC)$. 
If $T_0^{l_0}=\lambda T_1^{l_1}$ on $\tau(\CC)$ for some $\lambda\in\CC$, then again $\tau(\CC)$ is contained in a proper closed $T$-invariant subset $V(T_0^{l_0}-\lambda T_1^{l_1})$, and the curve can not be horizontal. 

Conversely, assume that the function $T_0^{l_0}/T_1^{l_1}$ is non-constant along $\tau(\CC)$. 
Let us consider an open subset $W_0$ in $X$ consisting of all points where each coordinate $T_{ij}$ is nonzero. Since the stabilizer in $T$ of a point on $W_0$ is trivial, all $T$-orbits in $W_0$ form a one-parameter family of orbits of codimension $1$ in $X$. The intersection of the curve $\tau(\CC)$ with $W_0$ is not contained in a $T$-orbit and thus it intersects generic $T$-orbits in $W_0$. This implies that the curve is horizontal.
\end{proof}

\begin{remark}
One may obtain examples of horizontal polynomial curves on a trinomial hypersurface $X$ as generic orbits of a regular action $\GG_a\times X\to X$, where $\GG_a$ is the additive group of the ground field $\CC$ and the action comes from a homogeneous locally nilpotent derivation of the algebra $\CC[X]$, cf.~\cite[Lemma~2]{Ar}.
\end{remark}

For a trinomial $T_0^{l_0}+T_1^{l_1}+T_2^{l_2}$ of Type~2, we let $d_i=\text{gcd}(l_{i1},\ldots,l_{in_i})$.

\begin{theorem} \label{Hor}
Let $X$ be a trinomial hypersurface of Type~2. The following conditions are equivalent.
\begin{enumerate}
\item[(i)]
The hypersurface $X$ is rational.
\item[(ii)]
The hypersurface $X$ admits a horizontal polynomial curve.
\item[(iii)]
Either at least one of the numbers $d_0,d_1,d_2$ is coprime with the others, or
$\text{gcd}(d_0,d_1)=\text{gcd}(d_0,d_2)= \text{gcd}(d_1,d_2)=2$.
\end{enumerate}
\end{theorem}

\begin{proof} Conditions~(i) and~(iii) are equivalent by~\cite[Proposition~5.5]{ABHW}.

\smallskip

Let us prove implication (ii)$\Rightarrow$(i). Assume that the hypersurface $X$ admits a horizontal polynomial curve $\tau$. Consider the rational quotient $\pi\colon X\to Y$, i.e. a rational morphism to an algebraic variety $Y$ with $\CC(Y)=\CC(X)^T$ defined by the inclusion $\CC(X)^T\subseteq\CC(X)$, see~\cite[Section~2.4]{PV} for more details. Then $Y$ is a curve and $\pi$ restricted to $\tau(\CC)$ gives rise to a dominant rational morphism from $\CC$ to $Y$. It shows that the curve $Y$ is rational. On the other hand, the variety $X$ contains an open subset isomorphic to $T\times Y'$, where $Y'$ is a curve birational to $Y$. This proves that the variety $X$ is rational.

\smallskip

We come to implication (iii)$\Rightarrow$(ii). Let us prove first that a rational Pham-Brieskorn surface $V_{p,q,r}=V(z_0^p+z_1^q+z_2^r)$ admits a horizontal polynomial curve. In this part we use a method proposed in~\cite{Ev} and fill a gap in the arguments given there.

Take $\epsilon\in\CC$, $\epsilon^q=-1$. We have
$$
(x^r+1)^p+(\epsilon(2x^r+1))^q+x^rl(x)=0
$$
for some polynomial $l(x)$. Assume first that $\text{gcd}(p,r)=\text{gcd}(q,r)=1$. Then there exist $u,v\in\ZZ_{>0}$ such that $vr-upq=1$. Let us take
$$
z_0=l(x)^{uq}(x^r+1), \quad z_1=\epsilon l(x)^{up}(2x^r+1), \quad z_2=l(x)^vx.
$$
This curve is horizontal because the polynomials $x^r+1$ and $x$ are coprime.

\smallskip

Now assume that $\text{gcd}(p,q)=\text{gcd}(p,r)=\text{gcd}(q,r)=2$ and $p\ge q\ge r$. Then
$p=2p_1$, $q=2q_1$, $r=2r_1$ with pairwise coprime $p_1,q_1,r_1$.

Consider an equation
\begin{equation}
l_0(x)^2w_0(x)^{2p_1}+l_1(x)^2w_1(x)^{2q_1}+l_2(x)^2w_2(x)^{2r_1}=0.
\label{eq1}
\end{equation}

Take positive integers $u_i,v_i$ such that
$$
u_1p_1-v_1q_1r_1=1, \quad u_2q_1-v_2p_1r_1=1, \quad u_3r_1-v_3p_1q_1=1.
$$
The polynomials
\begin{equation}
\begin{split}
z_0=l_0(x)^{u_1}l_1(x)&^{v_2r_1}l_2(x)^{v_3q_1}w_0(x), \quad
z_1=l_0(x)^{v_1r_1}l_1(x)^{u_2}l_2(x)^{v_3p_1}w_1(x), \\
& z_2=l_0(x)^{v_1q_1}l_1(x)^{v_2p_1}l_2(x)^{u_3}w_2(x)
\label{eq2}
\end{split}
\end{equation}
satisfy the equation $z_0^p+z_1^q+z_2^r=0$. Moreover, if $w_0(x)$ has a prime factor that does not appear in $w_1(x)$ and does not divide $l_0(x)l_1(x)l_2(x)$, then we obtain a horizontal curve.
Hence it suffices to find a solution of equation~(\ref{eq1}) that meet the latter condition.

\smallskip

We set $s(x)=\alpha(x^{2r_1}+1)^{p_1}$ with some $\alpha\in\CC$ and $m(x)=s(x)-(x^{2r_1}+2)^{q_1}$. Then
$$
(x^{2r_1}+2)^{2q_1}+4\alpha^2m(x)^2(x^{2r_1}+1)^{2p_1}=(s(x)-m(x))^2+(2s(x)m(x))^2=(s(x)+m(x))^2.
$$
Note that $m(0)=\alpha-2^{q_1}$. So the left hand side with $x=0$ equals
$$
2^{2q_1}+4\alpha^2(\alpha-2^{q_1})^2.
$$
Let $\alpha_0$ be a root of this polynomial. Then we have
\begin{equation}
(2\alpha_0 m(x))^2(x^{2r_1}+1)^{2p_1}+(x^{2r_1}+2)^{2q_1}+l_2(x)^2x^{2r_1}=0
\label{eq3}
\end{equation}
with some polynomial $l_2(x)$. Since $m(x)$ is coprime with both $x^{2r_1}+1$ and $x^{2r_1}+2$,
the polynomial curve coming from~(\ref{eq3}) via~(\ref{eq2}) is horizontal.
This completes the proof in the surface case.

\smallskip

Now we come to the case of a trinomial hypersurface of arbitrary dimension. It is well known that for all sufficiently large positive integers $c_i$ there exist positive integers $b_{i1},\ldots,b_{in_i}$ such that
$$
b_{i1}l_{i1}+\ldots+b_{in_i}l_{in_i}=c_id_i.
$$
We take sufficiently large pairwise coprime $c_0,c_1,c_2$ that are coprime with $d_0,d_1,d_2$,  find the corresponding $b_{ij}$, substitute $T_{ij}=z_i^{b_{ij}}$, and obtain
\begin{equation}
z_0^{c_0d_0}+z_1^{c_1d_1}+z_2^{c_2d_2}=0.
\label{h1}
\end{equation}
If the hypersurface~X is rational, surface~(\ref{h1}) is rational as well. We take a horizontal polynomial curve on this surface
$$
z_0=\phi_0(x), \quad z_1=\phi_1(x), \quad z_2=\phi_2(x).
$$
With $T_{ij}=\phi_i(x)^{b_{ij}}$ we obtain a polynomial curve on $X$. Let us check that this curve is horizontal. The rational invariants $T_i^{l_i}/T_j^{l_j}$ on this curve are equal to
$$
\frac{\phi_i(x)^{c_id_i}}{\phi_j(x)^{c_jd_j}}.
$$
This fraction is non-constant for some $i,j$ just because the curve on surface~(\ref{h1}) is horizontal. This completes the proof of Theorem~\ref{Hor}.
\end{proof}

\begin{remark}
By~\cite[Theorem~1.1(ii)]{HH}, a trinomial hypersurface of Type~2 is a factorial affine variety if and only if the numbers $d_0,d_1,d_2$ are pairwise coprime. In particular, every factorial trinomial hypersurface of Type~2 satisfies the conditions of Theorem~\ref{Hor}.
\end{remark}

\begin{remark}
In~\cite{AZ} we show that every irreducible simply connected curve on a toric affine surface $X$ is an orbit closure of an action $\GG_m\times X\to X$ of the multiplication group $\GG_m$ of the ground field. 
The results of this paper characterize existence of certain polynomial curves on affine hypersurfaces with a torus action of complexity one
\end{remark}

\begin{problem}
Let $X$ be a normal rational affine variety without non-constant invertible functions equipped
with a torus action $T\times X\to X$ of complexity one such that $\CC[X]^T=\CC$. Does $X$ admit a horizontal polynomial curve?
\end{problem}

One possible approach to this problem is to use Cox rings and total coordinate spaces, see~\cite[Section~1.6]{ADHL} for details. Namely, under our assumptions the variety $X$ has a finitely generated divisor class group $\Cl(X)$ and a finitely generated Cox ring $R(X)$. Moreover, the ring $R(X)$ is the quotient of a polynomial ring by an ideal generated by trinomials~\cite[Theorem~1.8]{HW}, and the total coordinate space $\overline{X}=\Spec(R(X))$ carries a torus action of complexity one. So one may try to construct a horizontal polynomial curve on $\overline{X}$ and then to project it to a horizontal polynomial curve on $X$ via the quotient morphism $\overline{X}\to X$. The difficulty with this approach is that the total coordinate space need not be rational, see~\cite[Example~5.12]{ABHW} and the following example.

\begin{example}
Consider the surface $V_{3,3,3}=V(z_0^3+z_1^3+z_2^3)$ in $\CC^3$. This surface is not rational and does not admit a horizontal polynomial curve. On the other hand, the quotient $X$ of the surface $V_{3,3,3}$ by the group $H=\ZZ/3\ZZ\times\ZZ/3\ZZ$ acting as
$$
(z_0,z_1,z_2)\mapsto (\epsilon_0z_0,\epsilon_1z_1,\epsilon_0\epsilon_1z_2), \quad
\epsilon_0^3=\epsilon_1^3=1,
$$
is a rational $\GG_m$-surface, see~\cite[Theorem~1.7]{HW}. One can check that the algebra $\CC[X]$ is generated by the functions
$$
a=z_0^3, \quad b=z_1^3, \quad c=z_2^3, \quad d=z_0z_1z_2^2, \quad e=z_0^2z_1^2z_2,
$$
and the formulas
$$
a=-(x+1)x^2, \quad b=-(x+1)x^3, \quad c=(x+1)^2x^2, \quad d=(x+1)^2x^3, \quad e=(x+1)^2x^4
$$
define a horizontal polynomial curve on $X$.
\end{example}

\section{Horizontal curves on trinomial hypersurfaces of Type~1}
\label{sec4}

In this section we study existence of horizontal polynomial curves on trinomial hypersurfaces of Type~1. For this we need the following important result, see \cite{St}, \cite{Mas} or \cite[Theorem~1.8]{Pr}. Given a polynomial $p(x)\in\CC[x]$, denote by $d_0(p(x))$ the number of its distinct roots (without counting multiplicities).

\medskip

{\bf The Mason-Stothers abc-Theorem.} \ Let $a(x)$, $b(x)$, $c(x)$ be three coprime polynomials, not all three constant. Assume that $a(x)+b(x)+c(x)=0$.
Then we have
$$
\max\{\deg a(x), \deg b(x), \deg c(x)\}\le d_0(a(x)b(x)c(x))-1.
$$

\medskip

Let $X$ be a trinomial hypersurface of Type~1.

\begin{theorem}\label{Hor1}
The following conditions are equivalent.
\begin{enumerate}
\item[(i)]
The hypersurface $X$ admits a horizontal polynomial curve.
\item[(ii)]
We have $l_{ij}=1$ for some $i=1,2$ and some $j=1,\ldots,n_i$.
\end{enumerate}
\end{theorem}

\begin{proof}
We begin with implication (ii)$\Rightarrow$(i). Renumbering, we may assume that $l_{11}=1$.
Then we let
$$
T_{11}=-x^{l_{21}}-1, \quad T_{12}=\ldots=T_{1n_1}=1, \quad
T_{21}=x, \quad T_{22}=\ldots=T_{2n_2}=1.
$$
This gives a horizontal polynomial curve on $X$.

\smallskip

We come to implication (i)$\Rightarrow$(ii). Let $T_{ij}(x)$ be a horizontal polynomial curve
on $X$. We let
$$
a(x)=T_{11}^{l_{11}}(x)\ldots T_{1n_1}^{l_{1n_1}}(x), \quad
b(x)=T_{21}^{l_{21}}(x)\ldots T_{2n_2}^{l_{2n_2}}(x), \quad c(x)=1.
$$
Denote by $m_{ij}$ the number of distinct roots of the polynomial $T_{ij}(x)$. By the
Mason--Stothers abc-Theorem, we have
$$
m_{11}l_{11}+\ldots+m_{1n_1}l_{1n_1}\le\deg a(x)\le d_0(a(x)b(x))-1\le
m_{11}+\ldots+m_{1n_1}+m_{21}+\ldots+m_{2n_2}-1
$$
and, similarly,
$$
m_{21}l_{21}+\ldots+m_{2n_2}l_{2n_2}\le
m_{11}+\ldots+m_{1n_1}+m_{21}+\ldots+m_{2n_2}-1.
$$
Summing up these two inequalities, we obtain
$$
m_{11}(l_{11}-2)+\ldots+m_{1n_1}(l_{1n_1}-2)+m_{21}(l_{21}-2)+\ldots+m_{2n_2}(l_{2n_2}-2)
\le -2.
$$
If all $l_{ij}\ge 2$, we come to a contradiction.
\end{proof}

\begin{remark}
Consider a trinomial hypersurface $X$ of Type~1 and let again $d_i=\text{gcd}(l_{i1},\ldots,l_{in_i})$. By~\cite[Corollary~3.5]{HW2}, the hypersurface $X$ is rational if and only if either at least one of $d_1,d_2$ equals 1, or $d_1=d_2=2$. Theorem~\ref{Hor1} shows that not every rational trinomial hypersurface of Type~1 admits a horizontal polynomial curve. Moreover, by \cite[Proposition~2.8]{HW}, a trinomial hypersurface of Type~1 is factorial if and only if either $n_il_{i1}=1$
for some $i=1,2$, or $d_1=d_2=1$. This shows that not every factorial trinomial hypersurface of Type~1 admits a horizontal polynomial curve. 
\end{remark}

\section{Schwartz-Halphen curves and platonic triples}
\label{sec5}

We keep the notation of the previous sections. For a polynomial curve $\tau\colon\CC\to X$, $\tau(x)=(T_{ij}(x))$, we let
$$
T_i^{l_i}(x):=T_{i1}(x)^{l_{i1}}\ldots\,T_{in_i}(x)^{l_{in_i}}.
$$

\begin{definition} A polynomial curve $\tau\colon\CC\to X$ on a trinomial hypersurface of Type~2
is called a \emph{Schwartz-Halphen curve} (an \emph{SH-curve} for short) if the polynomials $T_0^{l_0}(x), T_1^{l_1}(x), T_2^{l_2}(x)$ are coprime.
\end{definition}

In the case of a polynomial curve on the Pham-Brieskorn surface $V_{p,q,r}$, this condition means that the curve does not pass through the origin. 

\begin{lemma}
Any SH-curve on a trinomial hypersurface $X$ of Type~2 is horizontal. 
\end{lemma}

\begin{proof}
By Lemma~\ref{l01}, it suffices to show that the rational function $T_0^{l_0}/T_1^{l_1}$ is non-constant along any SH-curve. If this is not the case, we have that the polynomials $T_0^{l_0}(x)$ and $T_1^{l_1}(x)$ are proportional. Since they are coprime, these polynomials are constant. Then the polynomial $T_2^{l_2}(x)$ is constant as well, so the curve is constant, a contradiction.
\end{proof}

Lemma~\ref{aa} shows that the image $\tau(\CC)$ of an SH-curve $\tau\colon\CC\to X$ is contained in the smooth locus $X^{\reg}$. The following example shows that the converse statement does not hold.

\begin{example}
Consider the hypersurface $X$ given by
$$
T_{01}^3T_{02}+T_{11}^3T_{12}+T_{21}^2T_{22}=0
$$
and the curve $\tau\colon\CC\to X$ defined by
$$
T_{01}=x+1, \
T_{02}=x, \
T_{11}=x-1, \
T_{12}=x, \
T_{21}=x, \
T_{22}=-2(x^2+3).
$$
This curve is not an SH-curve, but all its points are smooth on $X$.
\end{example}

The following result generalizes Theorem~\ref{T2} to higher dimensions. In the proof we use the idea of the proof of \cite[Theorem~18.4]{Pr}. 

\begin{theorem} \label{SH}
Let $X$ be a trinomial hypersurface of Type~2. We assume that $l_{i1}\le\ldots\le l_{in_i}$ for $i=0,1,2$. Then the following conditions are equivalent. 
\begin{enumerate}
\item[(i)]
The hypersurface $X$ admits an SH-curve.
\item[(ii)]
The hypersurface $X$ admits a polynomial curve $\tau\colon\CC\to X^{\reg}$.
\item[(iii)] 
The triple $(l_{01},l_{11},l_{21})$ is platonic.
\end{enumerate}
\end{theorem}

\begin{proof}

Implication (i)$\Rightarrow$(ii) is observed above. For implication (iii)$\Rightarrow$(i), we assume that $(l_{01},l_{11},l_{21})$ is a platonic triple and let $T_{ij}(x)=1$ for all $i=0,1,2$ and all $1<j\le n_i$. By Theorem~\ref{T2}, the surface $V(T_{01}^{l_{01}}+T_{11}^{l_{11}}+T_{21}^{l_{21}})$ admits an SH-curve.

\smallskip

We come to implication (i)$\Rightarrow$(iii). Let $\tau\colon\CC\to X$ be an SH-curve. Without loss of generality we assume that $l_{01}\ge l_{11}\ge l_{21}\ge 2$. Let $a(x)=T_0^{l_0}(x)$, $b(x)=T_1^{l_1}(x)$, $c(x)=T_2^{l_2}(x)$.

Denote by $m_{ij}$ the number of pairwise distinct roots of the polynomial $T_{ij}(x)$. Then the Mason--Stothers abc-Theorem implies
\begin{align}
\sum_j l_{0j}m_{0j}\le \sum_j m_{0j}+\sum_j m_{1j}+\sum_j m_{2j} -1 \label{r1} \\
\sum_j l_{1j}m_{1j}\le \sum_j m_{0j}+\sum_j m_{1j}+\sum_j m_{2j} -1 \label{r2} \\
\sum_j l_{2j}m_{2j}\le \sum_j m_{0j}+\sum_j m_{1j}+\sum_j m_{2j} -1. \label{r3}
\end{align}
Summing up (\ref{r1}), (\ref{r2}), (\ref{r3}), we obtain
\begin{equation}
\sum_j l_{0j}m_{0j}+\sum_j l_{1j}m_{1j}+\sum_j l_{2j}m_{2j}\le
3\left(\sum_j m_{0j}+\sum_j m_{1j}+\sum_j m_{2j}\right) -3. \label{r4}
\end{equation}
Thus we have $l_{21}=2$. If $l_{11}=2$ then the triple $(l_{01},l_{11},l_{21})$ is platonic.

\smallskip

Assume that $l_{11}\ge 3$. Let $l_{21}=\ldots=l_{2s_2}=2$ and $l_{2j}\ge 3$ with $j>s_2$. We denote $l_{2j}$ and $m_{2j}$ with $j>s_2$ by $l_{2j}''$ and $m_{2j}''$ respectively, and $m_{21},\ldots,m_{2s_2}$ by $m_{2j}'$.

\smallskip

One obtains from~(\ref{r3}) the inequality
\begin{equation}
\sum_j m_{2j}'+\sum_j (l_{2j}''-1)m_{2j}''\le\sum_j m_{0j}+\sum_j m_{1j}-1. \label{r5}
\end{equation}
Then we have
\begin{equation}
\sum_j m_{2j}' \le \sum_j m_{0j}+\sum_j m_{1j}-1. \label{r6}
\end{equation}
It follows from~(\ref{r4}) and~(\ref{r6}) that
$$
\sum_j l_{0j}m_{0j}+\sum_j l_{1j}m_{1j}+\sum_j l_{2j}''m_{2j}''\le
3\left(\sum_j m_{0j}+\sum_j m_{1j}+\sum_j m_{2j}''\right) +\sum_j m_{2j}' -3 \le
$$
$$
\le 4\left(\sum_j m_{0j}+\sum_j m_{1j}\right) +3\sum_j m_{2j}'' -4.
$$
Thus we have
\begin{equation}
\sum_j l_{0j}m_{0j}+\sum_j l_{1j}m_{1j} \le 4\left(\sum_j m_{0j}+\sum_j m_{1j}\right) - 4.
\label{r7}
\end{equation}
This proves that $l_{11}=3$. Let $l_{11}=\ldots=l_{1s_1}=3$ and $l_{1j}\ge 4$ with $j>s_1$. We denote $l_{1j}$ and $m_{1j}$ with $j>s_1$ by $l_{1j}''$ and $m_{1j}''$ respectively, and $m_{11},\ldots,m_{1s_1}$ by $m_{1j}'$.

\smallskip

Then~(\ref{r2}) can be rewritten as
\begin{equation}
2\sum_j m_{1j}'+\sum_j (l_{1j}''-1)m_{1j}''\le\sum_j m_{0j}+\sum_j m_{2j}-1. \label{r8}
\end{equation}

Summing up (\ref{r5}) and (\ref{r8}), we obtain
\begin{equation}
\sum_j m_{1j}' \le 2\sum_j m_{0j}-2+\sum_j (2-l_{1j}'')m_{1j}''+\sum_j (2-l_{2j}'')m_{2j}''.
\label{r9}
\end{equation}

From (\ref{r3}) and (\ref{r9}) we get
\begin{equation}
\sum_j m_{2j}' \le 3\sum_j m_{0j}-3+\sum_j (3-l_{1j}'')m_{1j}''+\sum_j (3-2l_{2j}'')m_{2j}''.
\label{r10}
\end{equation}

Using (\ref{r1}), (\ref{r9}), (\ref{r10}), we obtain
$$
\sum_j l_{0j}m_{0j}\le \sum_j m_{0j}+\sum_j m_{1j}'+\sum_j m_{1j}''+\sum_j m_{2j}'+\sum_j m_{2j}''-1\le
$$
$$
\le \sum_j m_{0j}+2\sum_j m_{0j}-2+\sum_j(6-2l_{1j}'')m_{1j}''+3\sum_j m_{0j}-3
+\sum_j(6-3l_{2j}'')m_{2j}''-1\le 6\sum_j m_{0j}-6.
$$
This proves that $l_{01}\le 5$ and thus the triple $(l_{01},l_{11},l_{21})$ is platonic.

\smallskip

Finally let us prove implication (ii)$\Rightarrow$(i). Consider a curve $\tau\colon\CC\to X^{\reg}$ and assume that the polynomials $T_0^{l_0}(x)$, $T_1^{l_1}(x)$, $T_2^{l_2}(x)$ are not coprime. Let $L(x)$ be a linear form that divides all these three polynomials. There exist indices
$1\le j_s\le n_s$, $s=0,1,2$, such that $L(x)$ divides the polynomials $T_{sj_s}(x)$, $s=0,1,2$.

If at least one of the exponents $l_{sj_s}$ equals $1$, then the triple $(l_{01},l_{11},l_{21})$ is platonic and we use implication (iii)$\Rightarrow$(i). 

If all the exponents $l_{si_s}$ are greater than $1$, we consider the root $x=\alpha$ of the linear form~$L(x)$. By Lemma~\ref{aa}, the point $\tau(\alpha)$ is a singular point on $X$, a contradiction. 

\smallskip

This completes the proof of Theorem~\ref{SH}. 
\end{proof}

\begin{remark}
By~\cite[Section~2]{KZ}, an algebraic variety $X$ is said to be \emph{$\AA^1$-poor} if there exists a subvariety $Y$ of $X$ of codimension at least $2$ such that every polynomial curve on $X$ meets $Y$. 
Theorem~\ref{SH} implies that every trinomial hypersurface $X$ of Type~2 such that the triple  $(l_{01},l_{11},l_{21})$ is not platonic is $\AA^1$-poor. Indeed, any polynomial curve on $X$ meets the singular locus $Y$ of $X$. 
In particular, such hypersusfaces are \emph{rigid} in a sence that $X$ admits no non-trivial $\GG_a$-action or, equivalently, the algebra $\CC[X]$ admits no nonzero locally nilpotent derivation. 
Rigid factorial trinomial hypersurfaces of Type~2 are characterized in~\cite[Theorem~1]{Ar}. Moreover, an explicit description of the automorphism group of a rigid trinomial hypersurface can be found in~\cite[Theorem~5.5]{AG}.
\end{remark}

\begin{remark}
If $\tau\colon\CC\to X$ is a polynomial curve on a trinomial hypersurface $X$ of Type~1, then the polynomial $T_1^{l_1}(x)$ and $T_2^{l_2}(x)$ are coprime automatically. Thus every polynomial curve on $X$ is an SH-curve.
\end{remark}



\begin{thebibliography}{99}
%
\bibitem{Ar}
Ivan Arzhantsev. On rigidity of factorial trinomial hypersurfaces. Int. J. Algebra Comput. 26
(2016), no.~5, 1061-1070
%
\bibitem{AG}
Ivan Arzhantsev and Sergey Gaifullin. The automorphism group of a rigid affine variety.
Math. Nachr. 290 (2017), no.~5-6, 662-671
%
\bibitem{ABHW}
Ivan Arzhantsev, Lukas Braun, J\"urgen~Hausen and Milena Wrobel. Log terminal singularities,
platonic tuples and iteration of Cox rings. arXiv:1703.03627, 51 pages
%
\bibitem{ADHL}
Ivan Arzhantsev, Ulrich Derenthal, J\"urgen Hausen, and Antonio Laface. {\it Cox rings}. Cambridge
Studies in Adv. Math. 144, Cambridge University Press, New York, 2015
%
\bibitem{AZ}
Ivan Arzhantsev and Mikhail Zaidenberg. Acyclic curves and group actions on affine toric surfaces.
In: Affine Algebraic Geometry, Proc. of the Conference, Osaka, 3-6 March 2011, K.Masuda, H.Kojima,
T.Hishimoto (Eds.), World Sci. Publishing Co. Pte. Ltd. 2013, 1-41
%
\bibitem{BK}
Gottfried Barthel and Ludger Kaup. Topologie des surfaces complexes compactes singuli\'eres.
In: Sur la topologie des surfaces complexes compactes.  S\'em. Math. Sup. 80,
Presses Univ. Montreal, Que., 1982, 61-297
%
\bibitem{Ev}
Azriel Evyatar. On polynomial equations. Israel J. Math. 10 (1971), no.~3, 321-326
%
\bibitem{FZ}
Hubert Flenner and Mikhail Zaidenberg. Rational curves and rational singularities. Math. Z. 244
(2003), no.~3, 549-575
%
\bibitem{Hal}
Georges Henri Halphen. Sur la reduction des \'equations diff\'erentielles lin\'eaires aux formes
int\'egrables. M\'emoires pr\'esent\'es par divers savants $\grave a$  l'Academie des sciences de
l'Institut National de France, T.~XXVIII, N.~1, Paris, F.~Krantz, 1883; Oeuvres, Vol.~3, Paris,
1921, 1-260
%
\bibitem{HH}
J\"urgen Hausen and Elaine Herppich. Factorially graded rings of complexity one.
In: Torsors, \'etale homotopy and applications to rational points. London Math.
Soc. Lecture Note Series 405, 414-428, 2013
%
\bibitem{HHS}
J\"urgen Hausen, Elaine Herppich, and Hendrik S\"u{\ss}. Multigraded factorial rings and
Fano varieties with torus action. Doc. Math. 16 (2011), 71-109
%
\bibitem{HS}
J\"urgen Hausen and Hendrik S\"u{\ss}. The Cox ring of an algebraic variety with torus action.
Adv. Math. 225 (2010), no.~2, 977-1012
%
\bibitem{HW}
J\"urgen Hausen and Milena Wrobel. Non-complete rational T-varieties of complexity one.
Math. Nachr. 290 (2017), no.~5-6, 815-826
%
\bibitem{HW2}
J\"urgen Hausen and Milena Wrobel. On iteration of Cox rings. arXiv:1704.06523, 9 pages
%
\bibitem{KZ}
Shulim Kaliman and Mikhail Zaidenberg. Miyanishi's characterization of the affine 3-space does
not hold in higher dimensions. Ann. Inst. Fourier 50 (2000), no.~6, 1649-1669
%
\bibitem{Kl}
Felix Klein. \emph{Vorlesungen \"uber das Ikosaeder und die Aufl\"osung der Gleichungen vom
f\"unften Grade}. Teubner, Leipzig, 1884.
English transl.: Felix Klein. \emph{Lectures on the Icosahedron and the solution of equations
of fifth degree}, Dover, 1956
%
\bibitem{Mas}
R.C.~Mason. Diophantine equations over function fields. London Math. Soc,
Lecture Note Series 96, Cambridge University Press, Cambridge, 1984
%
\bibitem{PV}
Vladimir Popov and Ernest Vinberg. Invariant Theory.
Algebraic Geometry IV, Encyclopaedia Math. Sci.,
vol.~55, 123-284, Springer-Verlag, Berlin, 1994
%
\bibitem{Pr}
Victor Prasolov. \emph{Polynomials}. Algorithms and Computation in Math. 11,
Springer, Berlin, 2004, 301 pp.
%
\bibitem{Sch}
Hermann Schwarz. \"Uber diejenigen F\"alle, in welchen die Gaussische hypergeometrische Reihe eine algebraische Function ihres vierten Elementes darstellt.  J. Reine Angew. Math. 75 (1873), 292-335
%
\bibitem{St}
Walter Wilson Stothers. Polynomial identities and hauptmoduln. Quarterly J. Math. 32 (1981), no.~3, 349-370
%
\bibitem{We}
Wladimir Welmin. Solutions of the indeterminate equation $u^m+v^n=w^k$ (Russian), Mat. Sbornik 24 (1904), no.~4, 633-661
%
\end{thebibliography}
\end{document}